\documentclass[11pt]{article}
\usepackage{amssymb}
\usepackage{latexsym,amsmath,amssymb,amsfonts,epsfig,graphicx,cite}
\usepackage{eepic,color,colordvi,amscd}
\usepackage{mathrsfs}
\usepackage{amsthm}
\usepackage{graphicx}
\usepackage{enumerate}
\usepackage[pdfstartview=FitH,CJKbookmarks=true,bookmarksnumbered=true,bookmarksopen=true,
            colorlinks,pdfborder=001,linkcolor=black,anchorcolor=black,citecolor=black,
            ]{hyperref}
\allowdisplaybreaks[4]
\newtheorem{theo}{Theorem}[section]

\newtheorem{lem}[theo]{Lemma}

\newtheorem{cla}{Claim}

\newcommand{\F}{{\mathcal{F}}}

\begin{document}
\title{On the anti-Ramsey numbers of linear forests\thanks{This work is supported by is supported by the Youth Program of National Natural Science Foundation of China (No. 11901554)}}

\author{Tian-Ying Xie\\
		\small School of Mathematical Sciences\\
		\small University of Science and Technology of China\\
		\small Hefei, Anhui, 230026, China.\\
		\small  xiety@mail.ustc.edu.cn
		\and
		Long-Tu Yuan\\
		\small School of Mathematical Sciences\\
		\small East China Normal University\\
		\small Shanghai, 200241, China.\\
		\small ltyuan@math.ecnu.edu.cn\\
	}
\date{}
\maketitle
\bibliographystyle{plain}

\begin{abstract}
For a fixed graph $F$, the $\textit{anti-Ramsey number}$, $AR(n,F)$, is the maximum number of colors in an edge-coloring of $K_n$ which does not contain a rainbow copy of $F$. In this paper, we determine the exact value of anti-Ramsey numbers of linear forests for sufficiently large $n$, and show the extremal edge-colored graphs. This answers a question of Fang, Gy\H{o}ri, Lu and Xiao.
\end{abstract}

{{\bf Key words:}  Anti-Ramsey numbers, Linear forests.

\vskip 0.5cm

\section{Introduction}
In this paper, only finite graphs without loops and multiple edges will be considered. Let $K_n$ and $P_n$ be the clique and path on $n$ vertices, respectively. An \textit{even path} or \textit{odd path} is a path on $even$ or $odd$ number of vertices. A \textit{linear forest} is a forest whose  components are paths. For a given graph $G=(V(G),E(G))$, if $v\in V(G)$ is a vertex of $G$, let $N_G(v)$, $d_G(v)$ be the neighborhood and degree of $v$ in graph $G$ respectively, and if $U\subseteq V$, let $N_U(v)$ be the neighborhood of $v$ in $U$, and $N_G(U)$ be the common neighborhood of $U$ in $G$. An \textit{universal vertex} is a vertex in $V(G)$ which is adjacent to all other vertices in $V(G)$. Denote the minimum degree of $G$ by $\delta(G)$. Let $W$ and $U$ be two subsets of $V(G)$, denote the induced subgraph on $W$ of $G$ by $G[W]$, denote the subgraph of $G$ with vertex set $U\cup W$ and edge set $E(U,W)=\{uw\in E(G), u\in U, w\in W\}$ by $G[U,W]$.
An \textit{edge-colored graph} is a graph $G=(V(G),E(G))$ with a map $c:E(G)\rightarrow S$. The members in $S$ are called colors.
A subgraph of an edge-colored graph is \textit{rainbow} if its all edges have different colors.   The \textit{representing graph} of an edge-colored graph $G$ is a spanning subgraph of $G$ obtained by taking one edge of each color in $c$. Denoted by $\mathcal{R}(c,G)$ the family of representing subgraphs of an edge-colored graph $G$ with coloring $c$.

For a fixed graph $F$ and an integer $n$, the $\textit{anti-Ramsey number}$ of $F$ is the maximum number of colors in an edge-coloring of $K_n$ which does not contain $F$ as a rainbow subgraph, and denote it by AR$(n,F)$. The anti-Ramsey number was introduced by Erd\H os, Simonovits and S\'{o}s \cite{erdossim1975} in 1975. They determined the anti-Ramsey numbers of cliques when $n$ is sufficiently large. Later, in 1984, Simonovits and S\'{o}s \cite{simsos1984anpath} determined the anti-Ramsey number of paths.
\begin{theo}(Simonovits and S\'{o}s,\cite{simsos1984anpath})\label{thmpath}
Let $P_k$ be a path on $k$ vertices with $k\geq 2$. If $n$ is sufficiently large, then
\begin{equation*}
  \emph{AR}(n,P_k)=\binom{\lfloor \frac{k-1}{2}\rfloor-1}{2}+\left(\left\lfloor \frac{k-1}{2}\right\rfloor-1\right)\left(n-\left\lfloor \frac{k-1}{2}\right\rfloor  +1\right)+1+\varepsilon,
\end{equation*}
where $\varepsilon=1$ if $k$ is even and $\varepsilon=0$ otherwise.
\end{theo}
Moreover, they have given the unique extremal edge-colorings as following. Let $U$ be a vertex subset of $K_n$ with $|U|=\lfloor \frac{k-1}{2}\rfloor-1$, all the edges which are incident with $U$ have different colors, the all edges of $K_n[V(K_n)-U]$ colored by another one color if $k$ is odd or other two colors otherwise. Denoted by $\mathcal{C}_{P_k}(n)$ the family of above extremal edge-colorings of $K_n$ of $P_k$.

In 2004, Schiermeyer \cite{ingo2004matchings} determined the anti-Ramsey number of matchings for $n\geq 3t+3$.
\begin{theo}(Schiermeyer, \cite{ingo2004matchings})\label{thmmatching}
$\emph{AR}(n,tK_2)=\binom{t-2}{2}+(t-2)(n-t+2)+1$ for all $t\geq 2$ and $n\geq 3t+3$.
\end{theo}
And after that, Chen, Li and Tu \cite{chen2009matchings} and Fujita, Magnant and Ozeki \cite{Fujita2010survey} independently showed that AR$(n,tK_2)=\binom{t-2}{2}+(t-2)(n-t+2)+1$ for all $t\geq 2$ and $n\geq 2t+1$.

In 2016, Gilboa and Roditty \cite{gilboa2016anti} determined that for large enough $n$, the anti-Ramsey number of $L\cup tP_2$ and $L\cup kP_3$ when $t$ and $k$ are large enough and $L$ is a graph satisfying some conditions.

Very recently, Fang, Gy\H{o}ri, Lu and Xiao \cite{meilu2019anti} have given an approximate value of anti-Ramsey number of linear forests and determined the anti-Ramsey number of linear forests whose all components are odd paths.
\begin{theo}\label{thmlu}(Fang, Gy\H{o}ri, Lu and Xiao, \cite{meilu2019anti})
Let $F=\cup_{i=1}^{k} P_{t_i}$ be a linear forest, where $k\geq 2$, and $t_i\geq 2$ for all $1\leq i\leq k$. Then
\begin{equation*}
  \emph{AR}(n,F)=\left(\sum_{i=1}^k \left\lfloor \frac{t_i}{2} \right\rfloor -\epsilon\right)n+O(1),
\end{equation*}
where $\epsilon=1$ if all $t_i$ are odd and $\epsilon=2$ otherwise.
\end{theo}

For a given graph family $\F$, the \textit{Tur\'{a}n number} of $\F$ is the maximum number of edges of a graph on $n$ vertices which does not contain a copy of any graph in $\F$ as a subgraph, denote it by ex$(n,\F)$.

The anti-Ramsey problem of linear forest is strongly connected with the Tur\'{a}n number of linear forest. Hence, we introduce some results of the Tur\'{a}n numbers of paths and linear forests.

In 1959, Erd\H{o}s and Gallai showed the upper bound of the Tur\'{a}n number of $P_k$ as the following theorem.
\begin{theo}\label{eg}(Erd\H{o}s and Gallai, \cite{erdos1959path})
For any integers $k,n\geq1$, we have \emph{ex}$(n,P_k)\leq \frac{k-2}{2}n.$
\end{theo}
The Tur\'{a}n number of linear forest have been determined by Bushaw and Kettle \cite{Bushwa2011} and Lidicky, Liu and Palmer \cite{hliu2013turan} for sufficiently large $n$.

\begin{theo}(Bushaw and Kettle, \cite{Bushwa2011})
Let $k\cdot P_3$ be the vertex disjoint union of $k$ copies of $P_3$. Then for $n\geq 7k$, we have
\begin{equation*}
\emph{ex}(n,k\cdot P_3)=\binom{k-1}{2}+(k-1)(n-k+1)+\left\lfloor\frac{n-k+1}{2}\right\rfloor.
\end{equation*}
\end{theo}

\noindent{\bf Remark} Later, Yuan and Zhang \cite{Yuan} determined ex$(n,k\cdot P_3)$ for all values of $k$ and $n$.

\begin{theo}\label{thmex}(Lidicky, Liu and Palmer, \cite{hliu2013turan})
Let $F=\cup_{i=1}^{k} P_{t_i}$ be a linear forest, where $k\geq 2$ and $t_i\geq 2$ for all $1\leq i\leq k$. If at least one $t_i$ is not 3, then for $n$ sufficiently large,
\begin{equation*}
  \emph{ex}(n,F)=\binom{\sum_{i=1}^k \lfloor t_i/2\rfloor-1}{2}+\left(\sum_{i=1}^k \left\lfloor \frac{t_i}{2} \right\rfloor -1\right)\left(n-\sum_{i=1}^k \left\lfloor \frac{t_i}{2}\right\rfloor+1   \right)+c,
\end{equation*}
where $c=1$ if all $t_i$ are odd and $c=0$ otherwise. Moreover, the extremal graph is unique.
\end{theo}

The extremal graph in Theorem \ref{thmex}, denote by $G_F(n)$, is a graph on $n$ vertices with a vertex set $U$ of order $(\sum_{i=1}^{k}\lfloor t_i/2 \rfloor-1)$ such that all the vertices in $U$ are universal vertices and $G_F(n)-U$ contains a single edge if each $t_i$ is odd or $V(G)-U$ is an independent subset otherwise.

The anti-Ramsey numbers of linear forests which consist of odd paths are determined by Gilboa and Roditty  \cite{gilboa2016anti} for AR$(n,k\cdot P_3)$ and Fang, Gy\H{o}ri, Lu and Xiao \cite{meilu2019anti} otherwise.   In \cite{meilu2019anti}, they asked the following question: determining the exact value of anti-Ramsey number of a linear forest when it contains even paths. We will establish the following theorem.

From now on, let $F=\cup_{i=1}^{k} P_{t_i}$ be a linear forest with at least one $t_i$ is even, where $k\geq 2$, $t_i\geq 2$ for all $1\leq i\leq k$. Define \textit{$\mathcal{C}_F(n)$} to be a family of edge-colorings of $K_n$ with a subset $U$ of order $\sum_{i=1}^k \lfloor t_i/2 \rfloor -2$, the all edges which are incident with $U$ have different colors and the edges in $V(K_n)-U$ are colored by another $1+\varepsilon$ colors, where $\varepsilon=1$ if exactly one $t_i$ is even or $\varepsilon=0$ if at least two $t_i$ are even. (see Figure 1).

\begin{figure}[h]
\begin{center}
\includegraphics[width=0.70\textwidth]{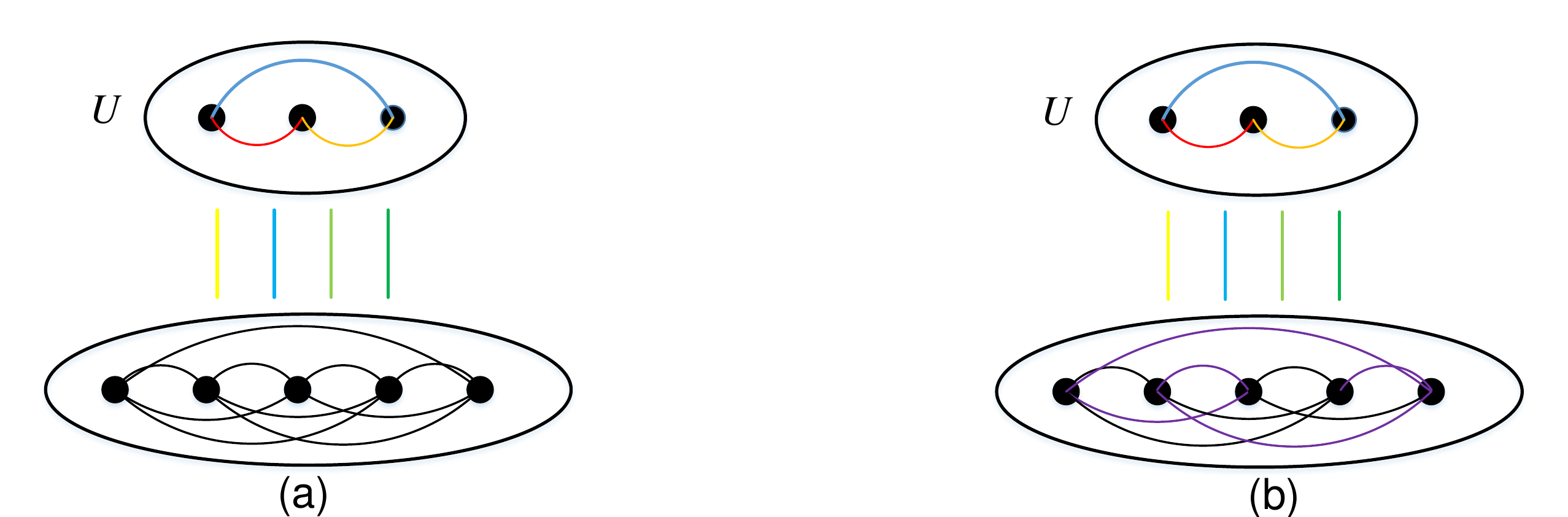}
\caption{$\mathcal{C}_F(n)$, (a) when $F$ contains at least two components with even vertices; (b) when $F$ contains exact one component with even vertices}
\end{center}
\end{figure}

\begin{theo}\label{thmmain}
There is a function $f(t_1,\ldots,t_k)$ such that  if $n\geq f(t_1,\ldots,t_k)$, then
$$\emph{AR}(n,F)=\binom{\sum_{i=1}^k \lfloor t_i/ 2\rfloor -2}{2}+ \left(\sum_{i=1}^k \left\lfloor \frac{t_i}{2} \right\rfloor -2\right)\left(n-\sum_{i=1}^k \left\lfloor \frac{t_i}{2} \right\rfloor +2 \right)+1+\varepsilon,$$
where $\varepsilon=1$ if exactly one $t_i$ is even or $\varepsilon=0$ if at least two $t_i$ are even. Moreover, the extremal edge-colorings must be in $\mathcal{C}_F(n)$.
\end{theo}

\section{Proof of Theorem~\ref{thmmain}}


First, we prove a useful lemma for the extremal problems of linear forests.
\begin{lem}\label{lem1}
Let $G$ be an $F$-free graph on $n$ vertices with $|V(F)|=f$. Let $F_1=\cup_{i\in L} P_{t_i}$ be a subgraph of $F$, where $L\subseteq [k]$ and $\sum_{i\in L}\lfloor t_i/2\rfloor\geq 2$, let $F_2=F-F_1$. If $G$ contains a copy of $F_1$ as a subgraph and
$$e(G)-\binom{|F_1|}{2}-\emph{ex}(n-|F_1|,F_2)\geq \left(\sum_{i\in L}\left\lfloor\frac{t_i}{2}\right\rfloor-\frac{7}{4}\right)n,$$
then any copy of $F_1$ in $G$ contains a subset $U$ of order $\sum_{i\in L}\lfloor t_i/2\rfloor-1$ with common neighborhood of size at least $2f^2+8f$ in $V(G)-U$.
\end{lem}
\begin{proof}
Let $G=(V,E)$ be an $F$-free graph on $n$ vertices. Assume that $G$ contains a copy of $F_1$ on subset $P$ and $p=|P|$. Let $t=\sum_{i\in L}\lfloor t_i/2\rfloor$. Since $G$ is $F$-free, $G[V-P]$ contains no copy of $F_2$. Hence $e(G[V-P])\leq \mbox{ex}(n-p,F_2)$ and the number of edges between $P$ and $V-P$ in $G$ is at least $e(G)-\binom{p}{2}-\mbox{ex}(n-p,F_2)\geq (t-\frac{7}{4})n$.
Let $n_0$ be the number of vertices in $V-P$ which have at least $t-1$ neighbors in $P$, this is,$$n_0=|\{v\in V-P: |N_P(v)|\geq t-1 \}|.$$ Then the number of edges between $V-P$ and $P$ is at most $n_0 p+(n-p-n_0)(t-2)$.
Hence
$$n_0 p+(n-p-n_0)(t-2)\geq\left(t-\frac{7}{4}\right)n.$$
So,
$$n_0\geq \frac{n/4+p(t-2)}{p-t+2}.$$
Since there are $\binom{p}{t-1}$ subsets with size $t-1$ in $P$, and $n$ is large enough, there is a subset $U$ of size $t-1$ in $P$ which has at least $n_0/\binom{p}{t-1}\geq 2f^2+8f$ common neighbors in $V-P$.
\end{proof}

\begin{lem}\label{U and W}
Let $K_n$ be a complete graph on $n$ vertices with an edge-coloring $c$. Let $U$ and $W$ be vertex disjoint subsets of $V(K_n)$ with $|U|=u$, $|W|=w$ and $u,w>0$. If there are two representing graphs $L^1_n$ and $L^2_n$ in $\mathcal{R}(c,K_n)$ such that $U$ has at least $s$ common neighbours in $L_n^1$ and $W$ has at least $s+su$ common neighbours in $L_n^2$, then there is a representing graph $L_n$ in $\mathcal{R}(c,K_n)$ such that $U$ and $W$ have at least $s$ common neighbours in $L_n$ respectively.
\end{lem}
\begin{proof}
Let $X$ with size $s$ and $Y$ with size $s+su$ be the common neighbours of $U$ in $L^1_n$ and the common neighbours of $W$ in $L^2_n$ respectively. Since there are $su$ colors between $X$ and $U$, there is a subset $Y^\prime$ of $Y$ with size at least $s$ such that the colors between $W$ and $Y^\prime$ are all different from the colors between $X$ and $U$. The result follows.
\end{proof}

The following lemma is trivial. We left its proof to the readers.
\begin{lem}\label{lemma for P2 P3}
For large $n$, $\emph{AR}(n,P_2\cup P_2)=1$ and $\emph{AR}(n,P_3\cup P_2)=2$.
\end{lem}

Let $t_1\geq t_2\ldots \geq t_k\geq 2$ and
\begin{equation*}
  f_F(n)=\binom{\sum_{i=1}^k \lfloor t_i/ 2\rfloor -2}{2}+ \left(\sum_{i=1}^k \left\lfloor \frac{t_i}{2} \right\rfloor -2\right)\left(n-\sum_{i=1}^k \left\lfloor \frac{t_i}{2} \right\rfloor +2 \right)+1+\varepsilon,
\end{equation*}
where $\varepsilon=1$ if exactly one $t_i$ is even and $\varepsilon=0$ if at least two $t_i$ are even. Let $F_0=F-P_{t_k}.$

We begin with a minimal degree version of the anti-Ramsey problem of linear forests.

\begin{lem}\label{lem2}
Let $c$ be an edge-coloring of $K_n$ which does not contain a rainbow copy of $F$ with at least $f_F(n)$ colors and $n\geq g(t_1,\ldots,t_k)$. If the minimum degree of each representing graph is at least $\sum_{i=1}^k \lfloor t_i/2 \rfloor-2$,
then the number of colors of $c$ is exactly $f_F(n)$, and the extremal edge-coloring must be in $\mathcal{C}_F(n)$.
\end{lem}
\begin{proof}
Let $c$ be an edge-coloring of $K_n$ on vertex set $V$ with at least $f_F(n)$ colors. Since $f_F(n)>$ex$(n,F_0)$ when $n$ is sufficiently large, each representing graph in $\mathcal{R}(c,K_n)$ contains a copy of $F_0$. Let $L^1_n\in \mathcal{R}(c,K_n)$ be a representing graph. Let $|V(F)|=f$.
By Lemma \ref{lemma for P2 P3}, we may assume that $F$ is not $P_2\cup P_2$ nor $P_3\cup P_2$, so $\sum_{i=1}^{k}\lfloor t_i/2\rfloor\geq 3$.

\begin{cla}\label{claim0}
There is a representing graph $L_n\in \mathcal{R}(c,K_n)$ such that it contains two disjoint vertex subset $U$ and $W$ \footnote{$W$ can be empty set} of order $u=|U|=\sum_{i=1}^{k-1}\lfloor t_i/2\rfloor-1$ and $w=|W|=\lfloor t_k/2\rfloor-1$ such that $|N_{L_n}(U)|\geq 2f+8$ and $|N_{L_n}(W)|\geq 2f+8$ if $W\neq \emptyset$.
\end{cla}
\begin{proof}
Since $F$ is not $P_2\cup P_2$ nor $P_3\cup P_2$,  we have $\sum_{i=1}^{k-1}\lfloor t_i/2\rfloor\geq 2$ and
\begin{align*}
  &  e(L^1_n)-\binom{|V(F_0)|}{2}-\mbox{ex}(n-|V(F_0)|,P_{t_k})\\
     \geq &\binom{\sum_{i=1}^k \lfloor t_i/ 2\rfloor -2}{2}+\left(\sum_{i=1}^k \left\lfloor \frac{t_i}{2} \right\rfloor -2\right)\left(n-\sum_{i=1}^k \left\lfloor \frac{t_i}{2} \right\rfloor +2 \right)+1+\varepsilon\\
    &-\binom{\sum_{i=1}^{k-1}t_i}{2}-\frac{t_k-2}{2}\left(n-\sum_{i=1}^{k-1}t_i\right)\\
      > &\left(\sum_{i=1}^{k-1} \left\lfloor \frac{t_i}{2} \right\rfloor -\frac{7}{4}\right)n.
\end{align*}
By Lemma~\ref{lem1}, one can find a subset $U$ of $P$ with  $|U|=\sum_{i=1}^{k-1}\lfloor t_i/2\rfloor-1$ which has common neighborhoods of size at least $2f^2+8f$ in $L^1_n$.

Now we consider the subgraph $L^1_n[V-U]$. Then we have
\begin{align*}
 e(L^1_n[V-U]) \geq& e(L^1_n)-\binom{\sum_{i=1}^{k-1}\lfloor\frac{t_i}{2}\rfloor-1}{2}-\left(\sum_{i=1}^{k-1}\left\lfloor\frac{t_i}{2}
 \right\rfloor-1\right)\left(n-\sum_{i=1}^{k-1}\left\lfloor\frac{t_i}{2}\right\rfloor+1\right) \\
      \geq &\binom{\sum_{i=1}^k \lfloor t_i/ 2\rfloor -2}{2}+\left(\sum_{i=1}^k \left\lfloor \frac{t_i}{2} \right\rfloor -2\right)\left(n-\sum_{i=1}^k \left\lfloor \frac{t_i}{2} \right\rfloor +2 \right)+1+\varepsilon\\
     &-\binom{\sum_{i=1}^{k-1}\lfloor t_i/2\rfloor-1}{2}-\left(\sum_{i=1}^{k-1}\left\lfloor\frac{t_i}{2}\right\rfloor-1\right)
     \left(n-\sum_{i=1}^{k-1}\left\lfloor\frac{t_i}{2}\right\rfloor+1\right)\\
        \geq &\binom{\lfloor t_k/2\rfloor-1}{2}+\left(\left\lfloor\frac{t_k}{2}\right\rfloor-1\right)\left(n-\sum_{i=1}^{k}\left\lfloor\frac{t_i}{2}\right\rfloor+2\right)+1+\varepsilon\\
        \geq &\binom{\lfloor (t_k-1)/2\rfloor-1}{2}+\left(\left\lfloor\frac{t_k-1}{2}\right\rfloor-1\right)\left(n-\sum_{i=1}^{k-1}\left\lfloor\frac{t_i}{2}\right\rfloor-\left\lfloor\frac{t_k-1}{2}\right\rfloor+2\right)+1+\varepsilon^\prime\\                                              = &\emph{AR}\left(n-\sum_{i=1}^{k-1}\left\lfloor\frac{t_i}{2}\right\rfloor+1,P_{t_k}\right).
 \end{align*}
where $\varepsilon^\prime=1$ if $t_k$ is even or $\varepsilon^\prime=0$ if $t_k$ is odd. Moreover, the equality holds if and only if $t_k$ is odd and at least two $t_i$'s are even. Thus, by Theorem~\ref{thmpath}, it is easy to see that $c\in \mathcal{C}_F(n)$ and we are done, or $K_n$ contains a rainbow copy of $P_{t_k}$ on subset $V-U$. So there is a representing graph $L_n^2\in\mathcal{R}(c,K_n)$ such that $L_n^2[V-U]$ contains a copy of $P_{t_k}$. If $t_k\leq 3$, let $W=\emptyset$; if $t_k\geq 4$, by Lemma~\ref{lem1}, there is a subset $W$ of $V-U$ of size $\lfloor t_k/2\rfloor-1$ has common neighborhoods of size at least $2f^2+8f$ in $L^2_n$. Now, by Lemma~\ref{U and W}, there is a representing graph $L_n$ satisfying the claim, the proof is completed.
\end{proof}

Since $\sum_{i=1}^k \lfloor t_i/2 \rfloor\geq 3$, we have $|U\cup W|\geq 1$. Let $T=V-U-W$.  We may choose $X\subseteq T$ and $Y\subseteq T$ be the set of common neighbours of $U$ and $W$  in $L_n$ respectively. By Claim~\ref{claim0}, we  have  $|X|\geq f+8$ when $|U|\geq 1$ and $|Y|\geq f+8$ when $|W|\geq 1$.

\begin{cla}\label{claim2}
Let $L$ be the set of common neighbours of $U\cup W$ in $L_n$. Then $|L|\geq f+2$.
\end{cla}
\begin{proof} If $U=\emptyset$ or $W=\emptyset$, then the claim is obviously true. Let $U\neq \emptyset$  and $W\neq \emptyset$, then $t_1\geq t_2\ldots \geq t_k\geq 4$. We consider the following two cases: (a) $t_k$ is even. Let $L_n^3$ be the graph obtained from $L_n$ by adding an edge $ab$ in $K_n[Y]$ and deleting the edge $cd$ in $L_n$ colored by $c(ab)$. Suppose that there are at least  three vertices in $X^\prime=X\setminus\{a,b,c,d\}$ which are not adjacent to all vertices of $W$. Since $\delta(L_n^3)\geq |U\cup W|$, there are three vertices of $X^\prime$, say $x_1,x_2$ and $x_3$, such that $x_i$ is adjacent to $y_i\in T$ for $i=1,2,3$. If two of $\{y_1,y_2,y_3\}$ are not belong to $\{a,b\}$, then we can find a copy of $F$ in $L^3_n$ easily. Moreover, the edges in $Y^\prime=Y\setminus\{a,b,c,d,y_1,y_2,y_3,x_1,x_2,x_3\}$  can not be colored by $c(ab)$. Otherwise, the graph obtained from $L_n$ by adding an edge colored by $c(ab)$ in $Y^\prime$ and deleting the edge $cd$ contains a copy of $F$. Now let $L_n^4$ be the graph obtained from $L_n$ by adding an edge $x_4x_5$ inside $Y^\prime$ and deleting the edge colored by $c(x_4x_5)$. Note that  at least two of $x_iy_i$ for $1\leq i\leq 3$ belong to $L_n^4$ and we delete at most one edge between $\{x_1,x_2,x_3\}$ and $U$ or between  $x_4x_5$ and $W$, we can easily find a copy of $F$ in $L_n^4$, a contradiction. Thus there are at most two vertices in $X^\prime$ which are not adjacent to all vertices of $W$. Hence, we have $|L|\geq f+8-6=f+2$. (b) $t_k$ is odd. The proof of this case is similar as case (a) and be omitted. The proof of the claim is completed.\end{proof}

\begin{cla}\label{claim3}
There are at most $1+\varepsilon$ colors in $c(K_n[T])$.
\end{cla}
\begin{proof} We only prove the claim for $\varepsilon=1$, since the case $\varepsilon=0$ is much easier. Then there is exactly one $t_i$ is even. Take the representing graph $L_n$ with $e(L_n[U\cup W,T])$ maximum and $|N_{L_n}(U\cup W)|\geq f$. That is if $z_1z_2$ with $z_1\in T$ and $z_2\in U\cup W$ is not an edge of $L_n$, then $z_1z_2$ is colored by $c(z_1^\prime z_2^\prime)$, where $z_1^\prime\in T$ and $z_2^\prime \in U\cup W$.

If $L_n$ is connected, suppose that there are at least three colors in $c(K_n[T])$.
Then by Lemma~\ref{lemma for P2 P3}, we can assume that $L_n[T]$ contains a copy of $P_{3}\cup P_2$.  Let $P_2=ab$ and $P_3=xyz$. We take $e(L_n[x,U\cup W] )$ as large as possible. Thus $x$ is adjacent to $U\cup W$.
And $ab$ is connected to $U\cup W\cup \{x,y,z\}$ by a path, let $P=wPw^\prime$ be the shortest path starting from $ab$ ending at $U\cup W\cup \{x,y,z\}$ with $w\in U\cup W\cup\{x,y,z\}$.
If $w\in U\cup W$, then $V(P)\cap \{x,y,z\}=\emptyset$. If $|U\cup W|=1$, then $\{u\}= U\cup W$. Thus we can easily find a copy of $F$ (Note that $F=P_5\cup P_2$, $F=P_4\cup P_3$ or $F=P_3\cup P_3\cup P_2$). Let $|U\cup W|\geq 2$. If $e(L_n[\{x\},U\cup W] )=1$, the edges between $x$ and $U\cup W$ are colored by the same color. We can take any edge of $K_n[\{x\},U\cup W]$ for $L_n$. If $e(L_n[\{x\},U\cup W] )\geq 2$, then there are at least two edges between $L_n[\{x\},U\cup W]$. Thus in both cases, we can take an edge $xu$ between $x$ and $U\cup W$ with $u \neq w $. Thus we can easily find a copy of $F$, a contradiction.
Now we may suppose $w\in \{x,y,z\}$. If $w\in \{y,z\}$. Then $xyPw^\prime$ or $xyzPw^\prime$ contains a copy of $P_4$ ending at $x$, so one can find a copy of $F$ in $L_n$.
If $w=x$. If $xPw\prime$ contains at least four vertices, one can find a copy of $F$ in $L_n$; otherwise, we may assume that $xa$ is an edge in $L_n$. If $z$ is adjacent to $U\cup W$ in $L_n$, then $zyxa$ is a copy of $P_4$ which is connected to $U\cup W$, then $L_n$ contains a copy of $F$; if for there is no edge between $U\cup W$ and $z$ in $L_n$, then we may add  an edge $zz^\prime$ with $z^\prime\in U\cup W$ delete the edge in $L_n$ colored by $c(zz^\prime)$. Thus the new representing graph contains a copy of $F$, a contradiction.

Assume that $L_n$ is disconnected. Let $C_1$ be the component of $L_n$ containing $U\cup W$, $Z=V-V(C_1)$ and $Q=V(C_1)-U\cup W$. By the similar argument above, $c(K_n[Q]$ contains at most two colors. Let $L_n^5$ be the graph obtained from $L_n$ by adding an edge $vv^\prime$ inside $L$ and deleting the edge colored by $c(vv^\prime)$. Since $L_n^5[V(C_1)]$ contains a copy of $F_0$, $L_n^5[Z]$ is $P_{t_k}$-free. So, we have $e(L_n^5)\leq \binom{|U\cup W|}{2}+|U\cup W||Q|+2+\frac{t_k-2}{2}|Z|< f_F(n)$, a contradiction. The claim is proved.\end{proof}

Since $e(L_n)\geq f_F(n)$, by Claim~\ref{claim3}, we have $c\in \mathcal{C}_F(n)$. The proof is completed.\end{proof}

\begin{proof}[Proof of Theorem \ref{thmmain}]
Let $c$ be an edge-coloring of $K_n$ contains no rainbow copy of $F$ with at least $f_F(n)$ colors and $n\geq f(t_1,\ldots,t_k)$, where $f(t_1,\ldots,t_k)\gg g(t_1,\ldots,t_k)$. Suppose that each representing graph in $\mathcal{R}(c,K_n)$  has minimum degree at least $\sum_{i=1}^{k}\lfloor t_i/2\rfloor-2$. Hence, by Lemma \ref{lem2}, we have the number of edge-coloring in $c$ is $f_F(n)$ and $c\in \mathcal{C}_F(n)$.

Now, we may assume that there is a representing graph $L_n\in \mathcal{R}(c,K_n)$ with $\delta(L_n)\leq \sum_{i=1}^{k}\lfloor t_i/2\rfloor-3$. So there is a vertex $u_n$ in $V$ with degree at most $\sum_{i=1}^{k}\lfloor t_i/2\rfloor-3$ in $L_n$. Let $G^n= K_n$ and $G^{n-1}= K_n- u_n$. Then $G^{n-1}$ is an edge-colored completed graph on $n-1$ vertices with at least $f_F(n-1)+1$ colors. If each representing  graph in $R(c,G^{n-1})$  has minimum degree at least $\sum_{i=1}^{k}\lfloor t_i/2\rfloor-2$, then similar as the argument above, we have $G^{n-1}$ contains a rainbow copy of $F$. Hence, there is a vertex $u_{n-1}$ in $G^{n-1}$ with degree at most $\sum_{i=1}^{k}\lfloor t_i/2\rfloor-2$. Thus we may construct a sequence of graphs $G^n,G^{n-1},\ldots,G^{n-\ell}$ such that the number of coloring of $G^{n-\ell}$ is at least $f_F(n-\ell)+\ell$ (note that $f(t_1,\ldots,t_k)\gg g(t_1,\ldots,t_k)$). Note that there are at most ${n-\ell \choose 2}$ colors in $G^{n-\ell}$, we  get a contradiction when $\ell$ is large.\end{proof}

\end{document}